\documentclass[12pt,a4paper]{amsart}
\usepackage{amsaddr}
\usepackage{amsmath,amscd}
\usepackage{latexsym,amsmath,amssymb,mathrsfs,setspace,enumerate,tikz}
\usepackage{tikz-cd}
\usepackage{graphicx}
\usepackage{ragged2e}
\usepackage{subcaption}
\usepackage[T1]{fontenc}
\usepackage[utf8]{inputenc}
\usepackage{lmodern}
\usepackage{amssymb}
\usepackage{multicol}
\usepackage{tikz-network}
\usetikzlibrary{automata,positioning}
\usepackage[english]{babel} 
\usepackage{blindtext}

\usepackage[colorlinks=true,
            linkcolor=blue,
            citecolor=green,
            urlcolor=magenta]{hyperref}
 \usepackage[a4paper,
             left=1.2in,
             top=1in,
             bottom=1in]{geometry}           

 \makeatletter
\def\section{\@startsection{section}{1}{\z@}%
  {-2.5ex \@plus -1ex \@minus -.2ex}%
  {1.3ex \@plus.2ex}%
  {\normalfont\large\bfseries\raggedright}}
\makeatother
\theoremstyle{plain}
\newtheorem{theorem}{Theorem}[section]

\newtheorem{proposition}[theorem]{Proposition}
\newtheorem{corollary}[theorem]{Corollary}
\theoremstyle{definition}
\newtheorem{definition}[theorem]{Definition}
\newtheorem*{claim}{Claim:}

\newtheorem{Remark}[theorem]{Remark}
\newtheorem{example}[theorem]{Example}

\usepackage{amsthm}
\newcommand\opn{\mathrel{\ooalign{$\subseteq$\cr
  \hidewidth\raise.225ex\hbox{$\circ\mkern.5mu$}\cr}}}

\makeatletter
\renewenvironment{proof}[1][\proofname]{%
  \par\vspace{0.1em} 
  \pushQED{\qed}%
  \normalfont \topsep0pt \trivlist
  \item[\hskip\labelsep\itshape #1\@addpunct{.}]\ignorespaces
}{%
  \popQED\endtrivlist\@endpefalse
}
\makeatother

\usepackage{tikz-cd}
\begin{document}
	\title{On the Bisections of a local Lie groupoid}
	
			\author[
			Navya K Nair and P. G. Romeo]{
			Navya K. Nair and P. G. Romeo}
			
	\address{Department of Mathematics\\
		Cochin University of Science and Technology(CUSAT)\\Kochi, Kerala, India, 682022}
	\email{$ navyapady@gmail.com, \,  romeo_-parackal@yahoo.com$}
	
\begin{abstract}
In this paper, we study the local Lie group structure associated with the space of admissible bisections of a local Lie groupoid over a compact manifold.
 We further investigate the relation of this local Lie group to the Lie algebra of sections of the associated Lie algebroid.
In addition, we prove that the globalizability of a local Lie groupoid implies the globalizability of its associated local Lie group of bisections.\\


\end{abstract}

\maketitle
\section{Introduction}

\paragraph{ Although    Lie algebroids are infinitesimal counterparts of Lie groupoids, unlike the classical correspondence between Lie algebras and Lie groups, not every Lie algebroid admits an integration by a global Lie groupoid. The integrability problem for Lie algebroids was systematically studied by Crainic and Fernandes in 2003 \cite{Crainic}, who showed that the obstruction to integrability is global in nature and is governed by monodromy groups.  To overcome these global obstructions, one is naturally led to consider local Lie groupoids, where the structure maps are defined only in a neighbourhood of the unit section. In fact, every Lie algebroid admits a local integration by a local Lie groupoid, as shown by Cabrera, M\u{a}rcu\c{t}, and Salazar \cite{Cabrera}. Thus, local Lie groupoids arise naturally in the integration theory of Lie algebroids and provide an important bridge between infinitesimal and global geometric structures}.

A fundamental problem in the theory of local Lie groupoids concerns globalization, namely, determining when a local Lie groupoid can be extended to a global Lie groupoid. The classical analogue of this problem arises in the theory of local Lie groups. In this context, Mal'cev proved the following characterization of globalizability \cite{Malcev, Olver}.

\begin{theorem}
    A connected local Lie group is globalizable if and only if it is globally associative.
\end{theorem}

In  \cite{Fernandes,Fernandes2} Rui Loja Fernandes, Daan Michiels and Yuxuan Zhang, established an analogous result for local Lie groupoids by revealing a deep connection between associativity and integrability of local structures as follows.

\begin{theorem}
    A strongly connected local Lie groupoid is globalizable if and only if it is globally associative.
\end{theorem}
      
It is also well known that, the theory of bisections forms an important bridge between Lie groupoids and infinite-dimensional Lie theory. In \cite{Schmeding} Schmeding and Wockel showed that, under suitable assumptions, the space of bisections of a Lie groupoid carries a natural infinite-dimensional Lie group structure and reflects several geometric and infinitesimal properties of the underlying groupoid.  Motivated by these developments, it is natural to investigate analogous questions for local Lie groupoids.

The main purpose of this paper is to study the local Lie group structure associated with the space of admissible bisections of a local Lie groupoid. More precisely, for a locally convex, locally metrizable local Lie groupoid over a compact manifold admitting an adapted local addition, we construct a local Lie group structure on the space of admissible bisections and 
obtained the following.
\begin{theorem}
Let $G$ be a locally convex, locally metrizable local Lie groupoid over a compact manifold $M$ admitting an adapted local addition. Then $Bis(G)$, the set of all admissible bisections of $G$ over $M$, forms a local Lie group.
\end{theorem}

The hypotheses are the natural local analogues of those imposed by Schmeding--Wockel in the global setting: local convexity and local metrizability ensure that the arrow space supports a well-behaved mapping space structure, the adapted local addition is the key technical condition enabling the manifold structure on \(C^\infty(M,G)\), and compactness of \(M\) guarantees that the section spaces carry a well-behaved locally convex topology.
We further investigate the relation between this local Lie group and the Lie algebra of sections of the associated Lie algebroid. We also study the globalization problem for local Lie groupoids and their associated local Lie groups of bisections, and establish conditions under which the globalizability of a local Lie groupoid implies the globalizability of its associated local Lie group of bisections, and vice versa.

\section{Preliminaries}
In this section, we recall the basic notions and conventions concerning local Lie groupoids and local Lie groups that will be used throughout this paper and also fix the notation and conventions adopted throughout. For background on finite-dimensional Lie groups and Lie groupoids we refer to \cite{Lee, Mackenzie} and the references therein. For infinite-dimensional manifold theory, we refer to \cite{Schmeding1}. 
We begin by recalling the notion of a local Lie groupoid.
\begin{definition} (cf. \cite{Fernandes2})
 A local Lie groupoid $G$ over a manifold ${M}$ is a manifold $G$ together with:
 \begin{itemize}
 \small
     \item a source ${\mathbf{\alpha}}$ and target maps ${\mathbf{\beta}}$ : ${G}\rightarrow {M}$ that are submersions;
     \item ${1}$: ${M}$ $\rightarrow$ ${G}$ (\textit{unit map}), maps $x \mapsto 1_x$, smooth;
     \item the multiplication $m: \mathcal{U} \rightarrow G$, submersion, where\\$\bigcup_{g \in G}\{(g, 1_{\alpha(g)}), (1_{\beta(g)}, g)\}$ $\subseteq$ $\mathcal{U}$ $\subseteq$ $G^{(2)}$ = $\{(g,h) \in {G\times G} : {\alpha}(g)= {\beta}(h)\}$;
     \item an inversion $i: \mathcal{V} \rightarrow \mathcal{V}$, a smooth map where $1(M) \subseteq \mathcal{V} \subseteq G$ such that $\mathcal{V}^{(2)} \subset \mathcal{U}$;
 \end{itemize}
 such that the following axioms hold:

\begin{itemize}
\item $\alpha (m(g,h))$ = $\alpha(h)$ and $\beta(m(g,h))$ = $\beta(g)$ $\forall$ $(g,h)$ $\in$ ${\mathcal{U}^{(2)}}$;
\item $m(g,m(h,k))$ = $m(m(g,h),k)$  if $(g,h), (h,k), (m(g,h),k), \\(g, m(h,k)) \in \mathcal{U}$ ;

\item $\alpha(1_x)$= $\beta(1_x)$ = $x$ \,\, $\forall$ $x$ $\in$ ${M}$;

\item $m(g, 1_{\alpha(g)})$ = $g$ and  $ m(1_{\beta(g)},g)$ = $g$ \,\, $\forall$ $g$ $\in$ ${G}$;
\item $\alpha(i(g))$= $\beta(g)$ and $\beta(i(g))$= $\alpha(g)$;
\item $ m(i(g),g)$ = $ 1_{\alpha(g)}$, $m(g,i(g))$ =  $1_{\beta(g)}$ for all $g \in \mathcal{V}$.
\end{itemize}

\end{definition}

Any Lie groupoid $G$ is a local Lie groupoid with $\mathcal{V}=G$ and $\mathcal{U}=G^{(2)}$.
Following is a nontrivial example.
   
\begin{example}\label{ex1}
Let $G=\{(x,y,a)\in S^2\times S^2\times \mathbb R:x+y\neq 0\}$. Then $G$ is a local Lie groupoid over $\mathbb{S}^{2}$ with source and target maps $\alpha(x,y,a)=y$ and $\beta(x,y,a)=x$ respectively and the unit map $1_x=(x,x,0)$. The multiplication of $(x,y,a),(y,z,b)\in G$ is defined whenever $x+z\neq 0$ and is given by 
\[
(x,y,a)\cdot (y,z,b)
=
(x,z,a+b+A(\Delta xyz)),
\]
where $A(\Delta xyz)$ denotes the signed area of the geodesic triangle determined by the points $x,y,z\in S^2$, and the inversion is given by
$i(x,y,a)=(y,x,-a)$.
\end{example}
The morphism between local Lie groupoids is defined as follows.
\begin{definition} (cf. \cite{Fernandes})
    Suppose that $G_1$ and $G_2$ are local Lie groupoids over $M_1$ and $M_2$ respectively. A  morphism  between them is a pair $(F,f)$ where $F: G_1 \rightarrow G_2$ and $f: M_1 \rightarrow M_2$ are smooth maps such that
    \begin{enumerate}
        \item $F\circ 1_{G_1}=1_{G_2}\circ f$.
        \item $f \circ \alpha_1$= $\alpha_2 \circ F$.\\ $f \circ \beta_1$= $\beta_2 \circ F$.
        \item $ (F \times F)(\mathcal{U}_1)\subseteq \mathcal{U}_2$ and $F \circ m_1  = m_2 \circ (F \times F)$ on $\mathcal{U}_1$
        \item $F(\mathcal{V}_1)\subseteq \mathcal{V}_2$ and $F \circ i_1 = i_2 \circ F$ on $\mathcal{V}_1$.
    \end{enumerate}
\end{definition}

\begin{definition}
A local Lie groupoid $G$ is called \emph{globally associative} if it is associative to every order $n \geq 3$.
\end{definition}

Next recall a local Lie group $(L,\mu,U, \mathrm{i}, V)$.
\begin{definition} (cf. \cite{Neeb})
    A smooth manifold L is called a local Lie group if there exists:
    \begin{itemize}
        \item a distinguished element $e \in L$, the identity element;
        \item a smooth product map $\mu: U \rightarrow L$ defined on an open subset $U$  such that $(\{e\}\times L)\bigcup (L \times \{e\})\subseteq U \subseteq L \times L;$
        \item a smooth inversion map $\mathrm{i}:V \rightarrow L$ defined on an open subset $V$ satisfying $e \in V\subseteq L$ such that $V\times \mathrm{i}(V) \subseteq U$ and $\mathrm{i}(V)\times V\subseteq U$;
    \end{itemize}
    all satisfying the following properties:
    \begin{itemize}
        \item identity: $\mu(e,x)=x=\mu(x,e) \,\, \forall x \in L$ ;
        \item inverse: $\mu(\mathrm{i}(x),x)=e=\mu(x,\mathrm{i}(x)) \,\,\forall x \in \mathrm{V}$;
        \item associativity: if $(x,y), (y,z), (\mu(x,y),z)$   and  $(x, \mu(y,z))\in U$, then $\mu(\mu(x,y),z)=\mu(x, \mu(y,z))$. 
    \end{itemize}
\end{definition}
Clearly any local Lie group is local Lie groupoid with objects the singleton: $M=\{*\}$.
     Given two local Lie groups 
$(L,\mu, U, i, V)$ and $(\tilde{L},\tilde{\mu}, \tilde{U}, \tilde{i},\tilde{V})$, 
a \textit{morphism} between them is a smooth map 
$\Phi: L \to \tilde{L}$ satisfying:
\begin{enumerate}
    \item $(\Phi \times \Phi)(U)\subseteq \tilde{U}$,
    
    \item $\Phi(V)\subseteq \tilde{V}$,
    
    \item $\Phi(e)=\tilde{e}$,
    
    \item  $  \Phi(\mu(g,h))
    =
    \tilde{\mu}(\Phi(g),\Phi(h))
    \quad \forall (g,h)\in U$,
    
    \item 
    
    $\Phi(i(g))
    =
    \tilde{i}(\Phi(g))
    \quad \forall g\in V$.
    
\end{enumerate}

\begin{definition}
A local Lie group $L$ is called \emph{globalizable} if there exists a local group homeomorphism
$\Phi:L\to N$
mapping $L$ onto a neighborhood
$e\in N\subseteq G$
of the identity of a global Lie group $G$.
\end{definition}

The following definition introduces local additions, which are used to construct manifold structures on spaces of local bisections.
\begin{definition} (cf. \cite{Schmeding}, 2.1)Let $Q,N$ be two smooth manifolds and $s \colon Q \to N$ be a surjective submersion (i.e. a smooth map whose differential is surjective at every point). A \emph{local addition adapted to $s$} is a local addition $\Sigma: U\opn TQ \rightarrow Q$ where $TQ$ denotes the tangent bundle of $Q$ and $U$ is an open subset of $TQ$
such that the fibres of $s$ are additively closed with respect to $\Sigma$, i.e.
$\Sigma(v_q) \in s^{-1}(s(q))$  for all $q \in Q$ and $v_q \in T_qs^{-1}(s(q)) $.
\end{definition} 

\section{Local group of bisections}

Recall that a bisection of a Lie groupoid $G\rightrightarrows M$ is a smooth map $\sigma: M \rightarrow G$ which is right inverse to $\alpha: G \rightarrow M$ and is such that $\beta \circ \sigma :M \rightarrow M$ is a diffeomorphism. The same definition applies to local Lie groupoids.

\begin{definition}
Let $G$ be a local Lie groupoid over $M$. 
A bisection of $G$ is a smooth section 
$\sigma : M \to G$ of $\alpha$ such that 
$\beta \circ \sigma : M \to M$ is a diffeomorphism.
\end{definition}
We say a local Lie groupoid $G$ over $M$ admits an adapted local addition if $G$ admits a local addition adapted to the source projection $\alpha$.
\begin{definition}
A local Lie groupoid $G$ over a base $M$ is called a locally convex, locally metrizable local Lie groupoid if both the arrow manifold $G$ and the object manifold $M$ are smooth manifolds modeled on locally convex, locally metrizable topological vector spaces.
\end{definition}
\begin{theorem}\label{theo1}
   Let $G$ be a locally convex, locally metrizable local Lie groupoid over a compact manifold $M$ admitting an adapted local addition. Then $Bis(G)$, the set of all bisections of $G$ over $M$, forms a local Lie group.
  \end{theorem}

\begin{proof}

The smooth manifold structure on $Bis(G)$ is constructed following \cite{Schmeding}. 
By Theorem 7.8(e) of \cite{Schmeding}, the map
\[
C^\infty(M,G)\to C^\infty(M,M),\qquad
f\mapsto \beta\circ f
\]
is smooth. Moreover, the space of sections
$\Gamma(M \leftarrow^\alpha G)$ 
is a submanifold of $C^\infty(M,G)$. Hence the restriction
\[
\beta^*:\Gamma(M \leftarrow^\alpha G)\to C^\infty(M,M),\qquad
f\mapsto \beta\circ f
\]
is smooth.
Since $\mathrm{Diff(M)}$ is an open submanifold of 
$C^\infty(M,M)$, it follows that
\[
Bis(G)=({\beta^*})^{-1}(\mathrm{Diff(M)})
\]
is an open submanifold of 
$\Gamma(M \leftarrow^\alpha G)$ and hence of $C^{\infty}(M,G)$ .
We now construct the local Lie group structure on $Bis(G)$.
The unit section
$ 
1:M\to G
$ 
is a bisection since
$ \alpha\circ 1=\beta\circ 1=\mathrm{id}_M.
$ 

Let $\mathcal V$ be a neighborhood of $1(M)$ in $G$
on which the inversion map is defined, and define
\[
\mathrm V
=
\{
\sigma\in Bis(G):
\sigma(M)\subseteq \mathcal V
\}.
\]
Then $1\in \mathrm V$.

Let $\mathcal U\subseteq G^{(2)}$ be the domain of multiplication of the local Lie groupoid $G$, and define
\[
\mathrm U
=
\left\{
(\sigma,\tau)\in Bis(G)\times  Bis(G):
(\sigma(\beta\circ\tau(x)),\tau(x))
\in \mathcal U
\ \forall x\in M
\right\}
\]
then
\[
(\{1\}\times Bis(G))
\cup
(Bis(G)\times\{1\})
\subseteq
\mathrm U
\subseteq
Bis(G)\times Bis(G).
\]

Define the product
$\mu:\mathrm U\to Bis(G)$ 
by
\[
\mu(\sigma,\tau)(x)
=
m(\sigma(\beta\circ\tau(x)),\tau(x)),
\qquad x\in M, where\,\, m\,\, is \,\,the \,\,local\,\,multiplication\,\, on\,\, G.
\]

For $\sigma\in \mathrm V$, define the inverse  $\mathrm{i}(\sigma)=\sigma^{-1}$  by
\[
\sigma^{-1}(x)
=
i\!\left(
\sigma\bigl((\beta\circ\sigma)^{-1}(x)\bigr)
\right),
\]

then the following identities hold:
\begin{itemize}
\setlength{\itemsep}{0.5em}

    \item 
    $\mu(1,f)=f=\mu(f,1)
    \qquad \forall f\in Bis(G)$ .

    \item $\mu(f^{-1},f)
    =
    1
    =
    \mu(f,f^{-1})
    \qquad \forall f\in \mathrm V$.

    \item $\mu(\mu(f,g),h)
    =\mu(f,\mu(g,h))$   whenever \small${
     (f,g),\ (g,h),\ (\mu(f,g),h),\ (f,\mu(g,h))
    \in \mathrm U}$.
    
\end{itemize}

\begin{claim}

$\mathrm V\times \mathrm V^{-1}\subseteq \mathrm U
\quad \text{and} \quad
\mathrm V^{-1}\times \mathrm V\subseteq \mathrm U$
\end{claim}
Let $(f,g^{-1}) \in \mathrm V\times \mathrm V^{-1}$. Then
$ (f(\beta \circ g^{-1}(x)), g^{-1}(x))
\in
\mathcal V^{(2)}
\subseteq
\mathcal U$,
 hence $(f,g^{-1}) \in \mathrm U$, therefore, $\mathrm V\times \mathrm V^{-1}\subseteq \mathrm U$.
The inclusion
$\mathrm V^{-1}\times \mathrm V\subseteq \mathrm U$
follows similarly.
\begin{claim}
The sets $\mathrm U$ and $\mathrm V$ are open.
\end{claim}

We have $\mathcal V\subseteq G$ open and $C^\infty(M,\mathcal V)
=
\{f\in C^\infty(M,G):f(M)\subseteq \mathcal V\}$ .
Since $\mathcal V$ is open and $M$ is compact, a standard result in convenient calculus implies that
$C^\infty(M,\mathcal V)$ 
is open in $C^\infty(M,G)$. Moreover,
$\mathrm V
=
Bis_{\mathcal V}(G)
=
Bis(G)\cap C^\infty(M,\mathcal V)$ ,
and hence $\mathrm V$ is open in $Bis(G)$, since the intersection of an open set with a submanifold is open in the submanifold.

To prove that $\mathrm U$ is open, define the smooth map
\[
\Psi:Bis(G)\times Bis(G)\to C^\infty(M,G^{(2)})
\]
by
\[
\Psi(\sigma,\tau)(x)
=
(\sigma(\beta\circ\tau(x)),\tau(x)).
\]
Since $\mathcal U\subseteq G^{(2)}$ is open and $M$ is compact,
$C^\infty(M,\mathcal U)$ 
is open in $C^\infty(M,G^{(2)})$. Therefore,
$\Psi^{-1}(C^\infty(M,\mathcal U))
=
\mathrm U$
is open in $Bis(G)\times Bis(G)$.

\begin{claim}
The maps $\mu$ and $\mathrm i$ are smooth.
\end{claim}

We have
$\mu:\mathrm U\to Bis(G)$ 
given by
\[\mu(\sigma,\tau)(x)
=
m(\sigma(\beta\circ\tau(x)),\tau(x)).
\]
This can be written as the composition
$\mu
=
m_*\circ \Psi'$ ,
where $\Psi'=\Psi|_{\mathrm U}$ and
$ m_*:C^\infty(M,\mathcal U)\to C^\infty(M,G),
\qquad
f\mapsto m\circ f$.
Since both $\Psi'$ and $m_*$ are smooth, it follows that $\mu$ is smooth.
Similarly, the inversion map
$\mathrm i:\mathrm V\to Bis(G)$ 
is given by
\[
\mathrm i(\sigma)(x)
=
i\!\left(
\sigma\bigl((\beta\circ\sigma)^{-1}(x)\bigr)
\right)=i_* (\sigma \circ (\beta \circ \sigma)^{-1})(x).
\]
Since composition and inversion in $\mathrm{Diff(M)}$ are smooth, it follows that $\mathrm i$ is smooth.
Therefore, $Bis(G)$ forms a local Lie group.
  
\end{proof}

\begin{Remark}
The hypotheses of Theorem~\ref{theo1} follow the framework of \cite{Schmeding}. Local convexity and local metrizability of $G$, together with compactness of $M$, ensure that $C^\infty(M,G)$  and the section space $\Gamma(M \leftarrow^\alpha G)$  carry well-behaved locally convex manifold structures. The adapted local addition is the key technical condition that makes the manifold structure on $C^\infty(M,G)$  possible in the first place.
\end{Remark}

\begin{example}\label{ex2}
   Consider the local Lie groupoid $G=\{(x,y,a)\in S^2\times S^2\times \mathbb R:x+y\neq 0\}$ given in example \ref{ex1}.
A bisection of $G$ is a smooth map $\sigma: \mathbb{S}^{2}\rightarrow G$ such that 
$\alpha\circ\sigma=\mathrm{id}_{S^2}$ 
and $\beta\circ\sigma:S^2\to S^2$
is a diffeomorphism.
Since
$\alpha(x,y,a)=y \, \, and \,\, \beta(x,y,a)=x$
every bisection is necessarily of the form
\[
\sigma(x)=(\phi(x),x,f(x)),
\]
where
$\phi:S^2\to S^2$
is a diffeomorphism and
$f:S^2\to \mathbb R$
is smooth, satisfying
$\phi(x)\neq -x,
\,\, \forall x\in S^2$.
i.e., $$Bis(G)=\{(\phi, id, f):\phi \in \mathrm{Diff}(\mathbb{S}^{2}), f \in C^{\infty}(\mathbb{S}^{2},\mathbb{R}), \phi(x)\neq -x\}$$
Now we examine the local Lie group structure of $Bis(G)$. 

Clearly, the unit map $1: \mathbb{S}^{2} \rightarrow G$, $x\mapsto(x,x,0)$ is  an element of $Bis(G)$.

Let
\[
\sigma(x)=(\phi(x),x,f(x))
\qquad\text{and}\qquad
\tau(x)=(\psi(x),x,g(x))
\]
be two bisections of $G$. Since
\[
\beta\circ\tau=\psi,
\]
we have
\[
\sigma(\beta\circ\tau(x))
=
(\phi(\psi(x)),\psi(x),f(\psi(x))),
\]

hence the product of bisections is given by
\[
(\sigma\star\tau)(x)
=
(\phi(\psi(x)),x,
f(\psi(x))+g(x)+A(\Delta(\phi(\psi(x)),\psi(x),x))).
\]

Similarly, the inverse of a bisection is
\[
\sigma^{-1}(x)
=
(\phi^{-1}(x),x,-f(\phi^{-1}(x))).
\]

Since composition and inversion in $\mathrm{Diff}(\mathbb{S}^2)$ are smooth, and the area function
$A(\Delta xyz)$ 
depends smoothly on $(x,y,z)$ locally, it follows that these operations define a local Lie group structure on $Bis(G)$.

\end{example}

\begin{theorem}\label{th0}
    Suppose that $G_1$ and $G_2$ are local Lie groupoids over the compact manifold $M$ and that $G_1$ and $G_2$ admits an adapted local addition. If $f: G_1 \rightarrow G_2$ is a morphism of local Lie groupoids over $M$ then
    $Bis(f): Bis(G_1) \rightarrow Bis(G_2)$, $\sigma \mapsto f \circ \sigma$ is a smooth morphism of local Lie groups.
\end{theorem}
\begin{proof}
  
Let
\[
f:G_1\to G_2
\]
be a morphism of local Lie groupoids. By Theorem 7.8(e) of \cite{Schmeding}, the induced map
\[
f_*:C^\infty(M,G_1)\to C^\infty(M,G_2),
\qquad
\gamma\mapsto f\circ \gamma
\]
is smooth. Since $Bis(G_1)$ is a submanifold of $C^\infty(M,G_1)$, the restriction
$f_*|_{Bis(G_1)}$ 
is smooth. Moreover, for each $\gamma\in Bis(G_1)$,
$f\circ \gamma\in Bis(G_2)$ .
Therefore, $Bis(f)=f_*|_{Bis(G_1)}$ 
defines a smooth map
\[
Bis(G_1)\to Bis(G_2).
\]

We now verify that $Bis(f)$ is a morphism of local Lie groups.

\begin{itemize}
    \setlength{\itemsep}{0.5em}

    \item  $ (f_*\times f_*)(\mathrm U_1)\subseteq \mathrm U_2$.

    Let $(\sigma,\tau)\in \mathrm U_1$. Since
     $(\sigma(\beta_1\circ\tau(x)),\tau(x))
    \in
    \mathcal U_1$
    for all $x\in M$ and $f$ is a morphism of local Lie groupoids, we obtain
    \[
    (f(\sigma(\beta_1\circ\tau(x))),f(\tau(x)))
    \in
    \mathcal U_2.
    \]
    Using $\beta_2\circ f=\beta_1$, it follows that
    \[
    (f\circ\sigma,f\circ\tau)\in \mathrm U_2.
    \]

    \item $f_*(\mathrm V_1)\subseteq \mathrm V_2$.

    Let $\sigma\in \mathrm V_1$. Since
    $\sigma(M)\subseteq \mathcal V_1$
    and $f(\mathcal V_1)\subseteq \mathcal V_2$, we obtain
    \[
    (f\circ\sigma)(M)\subseteq \mathcal V_2.
    \]
    Hence $f_*(\sigma)\in \mathrm V_2$.

    \item $f_*(e_1)=e_2$.

    Indeed,
    \[
    f_*(e_1)
    =
    f\circ 1_{G_1}
    =
    1_{G_2}\circ \mathrm{id}_M
    =
    e_2.
    \]

    \item $f_*(\mu_1(\sigma,\tau))
    =
    \mu_2(f_*(\sigma),f_*(\tau))
    \qquad
    \forall (\sigma,\tau)\in \mathrm U_1$.

    For $(\sigma,\tau)\in \mathrm U_1$,
    \begin{align*}
    f_*(\mu_1(\sigma,\tau))(x)
    &=
    f(m_1(\sigma(\beta_1\circ\tau(x)),\tau(x)))\\
    &=
    m_2(f(\sigma(\beta_1\circ\tau(x))),f(\tau(x)))\\
    &=
    m_2(f(\sigma(\beta_2 \circ f (\tau(x)))), f(\tau(x)))\\
    &=
    \mu_2(f\circ\sigma,f\circ\tau)(x).
    \end{align*}

    \item $f_*(\mathrm i_1(\sigma))
    =
    \mathrm i_2(f_*(\sigma))
    \qquad
    \forall \sigma\in \mathrm V_1$.

    For $\sigma\in \mathrm V_1$,
    \begin{align*}
    f_*(\mathrm i_1(\sigma))(x)
    &=
    f(i_1(\sigma((\beta_1\circ\sigma)^{-1}(x))))\\
    &=
    i_2(f(\sigma((\beta_1\circ\sigma)^{-1}(x))))\\
    &=i_2 \circ f(\sigma(\beta_2 \circ f \circ \sigma)^{-1}(x))\\
    &=
    \mathrm i_2(f\circ\sigma)(x).
    \end{align*}
 
\end{itemize}

Hence $Bis(f)$ is a smooth morphism of local Lie groups
    
\end{proof}

\begin{Remark}
Given a compact manifold $M$, the map $Bis$ can be regarded as  a functor   $Bis : \textit{locLiegrpds}^\Sigma_M \rightarrow \textit{locLiegrps}$, where $\textit{locLiegrpds}^\Sigma_M$ denotes the category whose objects are locally convex and locally metrisable local Lie groupoids over $M$ that admit an adapted local addition and $\textit{locLiegrps}$ denotes the category of local Lie groups.
\end{Remark}
\section{\textbf{Lie algebra of the bisection local group}}
Next we proceed to construct the Lie algebra of the local Lie group of bisections of a local Lie groupoid $\mathcal{G}$ and investigate whether it is isomorphic to the space of sections of the Lie algebroid $L(\mathcal{G})$, as in the Lie groupoid case. Readers may refer to \cite{Fernandes, Neeb} for the infinitesimal theory of local Lie groupoids and local Lie groups.

\begin{proposition} \label{pro1}
Let $G$ be a locally convex, locally metrizable local Lie groupoid over a compact manifold $M$, that admits adapted local addition. Then for each $x\in M$ there exists a  neighborhood $V\subseteq Bis(G)\times G$ of $(1,1_x)$  and a map 
$\gamma: V \rightarrow G$, $\gamma(\sigma,g)= m(\sigma(\beta(g)),g)$, which is smooth.
\end{proposition}

\begin{proof}
Define
\[
F:Bis(G)\times G\to G^{(2)},
\qquad
(\sigma,g)\mapsto (\sigma(\beta(g)),g).
\]
The map $F$ is smooth.

For each $x\in M$, we have
\[
(1,1_x)\in Bis(G)\times G
\]
and
\[
F(1,1_x)=(1_x,1_x)\in \mathcal U,
\]
where $\mathcal U\opn G^{(2)}$ denotes the domain of multiplication of the local Lie groupoid $G$. Since $\mathcal U$ is open, the set
\[
V:=F^{-1}(\mathcal U)
\]
is an open neighborhood of $(1,1_x)$ in $Bis(G)\times G$ satisfying
$F(V)\subseteq \mathcal U$.
Now define
\[
\gamma:=m\circ F|_V:V\to G.
\]
Since both $F|_V$ and $m$ are smooth, it follows that $\gamma$ is smooth.
\end{proof}

Following Remark 4.1 of \cite{Schmeding} in the Lie groupoid case, we obtain an isomorphism
\[
\varphi_G:T_1Bis(G)\to \Gamma(L(G)),
\]
given by
\[
[t\mapsto \eta(t)]
\mapsto
\bigl(
m\mapsto [t\mapsto \eta(t)(m)]
\bigr),
\]
where $\Gamma(L(G))$ denotes the space of sections of the Lie algebroid $L(G)$ associated with the local Lie groupoid $G$ and $T_1Bis(G)$ denotes the tangent space of  $Bis(G)$ at the identity element  $1$, which is the Lie algebra of the local Lie groupoid $Bis(G)$). Here tangent vectors are identified with equivalence classes of smooth curves.

\begin{proposition} \label{pro2}
    Let $G$ be a local Lie groupoid  over $M$ satisfying the assumptions of Theorem \ref{theo1} and let $X$ be an element of  $T_1Bis(G)$ and $0$ be the zero section of the tangent bundle  $TG \rightarrow G$. Then the vector fields 
    $\overrightarrow{\varphi_G(X)}$ and $X^{\rho}\times0$ are $\gamma-$related in the neighborhood $V$, i.e. the following diagram commutes.
    
\[\begin{tikzcd}
TBis(G)\times TG \arrow[r, "T\gamma"] & TG \\
V\subseteq Bis(G)\times G \arrow[u, "X^{\rho }\times 0"] \arrow[r, "\gamma"] & G \arrow[u, "\overrightarrow{\varphi_G(X)}"']
\end{tikzcd}\]

\end{proposition}

\begin{proof}
   
Represent an element $X\in T_1Bis(G)$ by the equivalence class $[\eta]$ of a smooth curve
$\eta:(-\varepsilon,\varepsilon)\to Bis(G)$ 
satisfying
$\eta(0)=1\,\, \text{and} \,\,
 \eta'(0)=X$.
Define
\[
\hat{\eta}:(-\varepsilon,\varepsilon)\times M\to G,
\qquad
(t,m)\mapsto \eta(t)(m).
\]
By Theorem 7.8(d) of \cite{Schmeding}, the map $\hat{\eta}$ is smooth.
For $\psi\in Bis(G)$, consider the right translation map
\[
\rho_\psi:U^\psi\to Bis(G),
\qquad
\sigma\mapsto \mu(\sigma,\psi),
\]
where $U^\psi
=
\{
\sigma\in Bis(G):
(\sigma,\psi)\in \mathrm U
\}$ .
Let $\eta_\psi:=\rho_\psi\circ \eta|_{\eta^{-1}(U^\psi)}$ .
Then $\eta_\psi$ is a smooth curve
\[
\eta^{-1}(U^{\psi})=(-r,r)\to Bis(G),
\qquad
t\mapsto \mu(\eta(t),\psi).
\]

Hence the associated map
$\widehat{\eta_\psi}:(-r,r)\times M\to G$
given by
\[
(t,x)
\mapsto
\eta_\psi(t)(x)
=
m(\hat{\eta}(t,\beta(\psi(x))),\psi(x))
\]
is smooth.
By definition of the right-invariant vector field associated with $X$, we have
\[
X^\rho(\psi)
=
T_1\rho_\psi(X)
\in
T_\psi Bis(G).
\]
Moreover, this can be represented by $[\eta_{\psi}].$ i.e.,
$X^\rho(\psi)
=
[t\mapsto \eta_\psi(t)]$ .
Now let $(\psi,g)\in V\subseteq Bis(G)\times G$.
Then
\begin{align*}
    \overrightarrow{\varphi_G(X)}\circ\gamma(\psi, g)&=\overrightarrow{\varphi_G(X)}(m(\psi(\beta(g)), g))\\
    &=T_{\tiny{\beta(\psi(\beta(g)))}}R_{\tiny{m(\psi(\beta(g)),g)}}\varphi_G(X)(\beta(\psi(\beta(g))))\\
    &=T_{\tiny{\beta(\psi(\beta(g)))}}R_{\tiny{m(\psi(\beta(g)),g)}}(t \mapsto \hat{\eta}(t, \beta(\psi(\beta(g))))\\
    &=(t \mapsto m(\hat{\eta}(t, \beta(\psi(\beta(g)))), m(\psi(\beta(g)),g))\\
    &= (t \mapsto m(m(\hat{\eta}(t, \beta(\psi(\beta(g)))),\psi(\beta(g))),g))\\
    &= (t \mapsto m(\hat{\eta_\psi}(t,\beta(g)),g))\\
     &= (t \mapsto m(\rho_\psi \circ \eta(t)(\beta(g)),g))\\
  &= (t \mapsto \gamma(\rho_\psi \circ \eta(t),g))\\
  &= T\gamma(X^{\rho}(\psi),0_g)\\
  &= T\gamma(X^{\rho}\times 0)(\psi,g).\\
\end{align*}
Here the map $R_g$ denotes the right multiplication by the element $g$ in the local Lie groupoid $G$ which is a map from $\mathcal{U}_g\subseteq \alpha^{-1}(\beta(g))$ to $\alpha^{-1}(\alpha(g))$, where $\mathcal{U}_g=\{g' \in G: (g',g)\in \mathcal{U}\}$. 

The product $m(\hat{\eta}(t, \beta(\psi(\beta(g)))), m(\psi(\beta(g)),g))$ and $m(\hat{\eta}(t, \beta(\psi(\beta(g)))), m(\psi(\beta(g))) $ 
exist by the same argument as in proposition \ref{pro1}.
\end{proof}

\begin{theorem}
    Let $G$ be a local Lie groupoid  over $M$ as above, then the morphism  $\varphi_G : L(Bis(G))\rightarrow \Gamma(L(G))$ of topological vector spaces is an anti-isomorphism of Lie algebras.
\end{theorem}
\begin{proof}
The proof is similar to that of Theorem 4.4 of \cite{Schmeding}. 
By Proposition \ref{pro2}, for $X,Y\in T_1 Bis(G)$, the vector fields
$\overrightarrow{\varphi_G(X)}
\,\, \text{and} \,\,
\overrightarrow{\varphi_G(Y)}$ 
are $\gamma$-related to $X^\rho\times 0
\,\, \text{and} \,\,
Y^\rho\times 0$,
respectively, on an open neighborhood
$V\subseteq Bis(G)\times G
$ of $(1,1_x)$. Hence,
\[
T\gamma\circ([X^\rho,Y^\rho]\times 0)
=
[\overrightarrow{\varphi_G(X)},
\overrightarrow{\varphi_G(Y)}]\circ\gamma,
\]
therefore,
\[
-\varphi_G([X,Y])(x)
=
[\varphi_G(X),\varphi_G(Y)](x).
\]

\end{proof}
\section{Globalization of Local Structures}
A local Lie groupoid is called \textit{globalizable} if it is a restriction of an open neighborhood of the unit section in a Lie groupoid. In this section, we investigate the relationship between the globalizability of the local Lie groupoid $G$ and the globalizability of its associated local Lie group of bisections. More precisely, we study whether the local Lie group $Bis(G)$ admits a globalization and how this globalization is related to the Lie group of bisections $Bis(\widetilde{G})$ of the 
globalization $\widetilde{G}$.

Recall that Mal'cev's theorem for local Lie groups states that a connected local Lie group, in the sense of \cite[Definition 13]{Olver}, is globalizable if and only if its multiplication is globally associative (cf.~\cite{Olver,Malcev}).
The corresponding globalization theorem for local Lie groupoids was established in \cite{Fernandes}.

\begin{theorem}[{\cite{Fernandes},Theorem 4.1}]\label{Th2}
 A strongly connected local Lie groupoid is globalizable if and only if its multiplication is globally associative.
\end{theorem}
Here strongly connectedness is understood in the sense of \cite{Fernandes}, Definition $2.16$.
Suppose that the local Lie groupoid $G$ is globally associative, i.e. $n-$associative for all $n\geq3$.
Clearly, \( Bis(G) \) is $3$-associative. Indeed, for
\(\sigma_1,\sigma_2,\sigma_3 \in Bis(G)\), we have
\begin{align*}
\mu(\mu(\sigma_1,\sigma_2),\sigma_3)(x)
&=
m\Big(
m\big(
\sigma_1(\beta(\sigma_2(\beta(\sigma_3(x))))),
\sigma_2(\beta(\sigma_3(x)))
\big),
\sigma_3(x)
\Big) \\
&=
m\Big(
\sigma_1(\beta(\sigma_2(\beta(\sigma_3(x))))),
m\big(
\sigma_2(\beta(\sigma_3(x))),
\sigma_3(x)
\big)
\Big) \\
&=
\mu(\sigma_1,\mu(\sigma_2,\sigma_3))(x),
\end{align*}
where the second equality follows from the associativity of the multiplication \(m\) on \(G\).
Now let
\[
P(\sigma_1,\dots,\sigma_n)
\quad \text{and} \quad
Q(\sigma_1,\dots,\sigma_n)
\]
be two defined parenthesized products of
$\sigma_1,\dots,\sigma_n \in Bis(G)$.
Since the multiplication on $Bis(G)$ is induced pointwise from the multiplication on $G$, for each $x \in M$, the evaluations
\[
P(\sigma_1,\dots,\sigma_n)(x)
\quad \text{and} \quad
Q(\sigma_1,\dots,\sigma_n)(x)
\]
yield two defined parenthesized products in $G$. As $G$ is globally associative, these products coincide whenever they are defined. Hence,
\[
P(\sigma_1,\dots,\sigma_n)(x)
=
Q(\sigma_1,\dots,\sigma_n)(x)
\]
for all $x \in M$. Therefore,
\[
P(\sigma_1,\dots,\sigma_n)
=
Q(\sigma_1,\dots,\sigma_n),
\]
showing that $Bis(G)$ is globally associative.

This observation is summarized in the following theorem.

\begin{theorem}\label{th3}
Let $G$ be a locally convex, locally metrizable local Lie groupoid over a compact manifold $M$ admitting an adapted local addition. If $G$ is globally associative, then the local Lie group of bisections $Bis(G)$ is also globally associative.
\end{theorem}
Combining Mal'cev's theorem for local Lie groups with its analogue for local Lie groupoids (Theorem \ref{Th2}), we obtain the following corollary.
  
\begin{corollary}
Along with the assumptions of the previous theorem, suppose that $G$ is globalizable and that the local Lie group of bisections $Bis(G)$ is connected. Then $Bis(G)$ is globalizable.
\end{corollary}

\begin{example}
Consider the local Lie groupoid  $G=\{(x,y,a)\in \mathbb{S}^{2}\times \mathbb{S}^2\times \mathbb{R}/4\pi\mathbb{Z}|x+y\neq0\}$ over $\mathbb{S}^2$ with source and target maps $\alpha(x,y,a)=y$ and $\beta(x,y,a)=x$, Unit map $1_x=(x,x,0)$ and inversion $i(x,y,a)=(y,x,-a)$. The multiplication of $(x,y,a),(y,z,a')\in G$ is defined whenever $x+z\neq 0$ and is given by $$(x,y,a)(y,z,a')=(x,z,a+a'+A(\Delta xyz))$$ where  $A(\Delta xyz)\in \mathbb{R}/4\pi$ is the signed area of the spherical triangle $\Delta xyz$. This defines a globalizable local Lie groupoid over $S^2$ (cf.~\cite{Fernandes}, Example 3.2).
We now construct the local Lie group $Bis(G)$ and investigate whether it is globalizable.

A bisection of $G$ is a smooth map
$\sigma:\mathbb S^2\to G$
such that
$\alpha\circ\sigma=\mathrm{id}_{\mathbb S^2}$ 
and
$\beta\circ\sigma:\mathbb S^2\to \mathbb S^2$
is a diffeomorphism.
Every bisection is necessarily of the form
\[
\sigma(x)=(\phi(x),x,f(x)),
\]
where $\phi:\mathbb S^2\to \mathbb S^2$
is a diffeomorphism and
$f:\mathbb S^2\to \mathbb R/4\pi\mathbb Z$
is smooth, satisfying
$\phi(x)+x\neq 0
\qquad
\forall x\in \mathbb S^2$. Conversely, every such pair $(\phi,f)$ determines a bisection of $G$.
Let $\sigma(x)=(\phi(x),x,f(x))
\,\, \text{and}\,\,
\tau(x)=(\psi(x),x,g(x))$ 
be two bisections of $G$. By the same argument as in Example \ref{ex2} the product of bisections is given by
\[
(\sigma\star\tau)(x)
=
(\phi(\psi(x)),x,
f(\psi(x))+g(x)+A(\Delta(\phi(\psi(x)),\psi(x),x))).
\]
and the inverse of a bisection is       
$\sigma^{-1}(x)
=
(\phi^{-1}(x),x,-f(\phi^{-1}(x)))$.
These operations provide a local Lie group structure on $Bis(G)$, and the multiplication on $Bis(G)$ is globally associative. In particular, let
$\sigma=(\phi,id,f),\,\,
\tau=(\psi,id,g),\,\,
\rho=(\chi,id,h)$ 
be three bisections of $G$.
Now consider
\[
(\sigma\star\tau)\star\rho
\qquad\text{and}\qquad
\sigma\star(\tau\star\rho).
\]
The diffeomorphism components clearly coincide and are equal to
$\phi\circ\psi\circ\chi$.

The difference between the scalar components is of the form
\[
A(\Delta pqr)
+
A(\Delta prs)
-
A(\Delta qrs)
-
A(\Delta pqs)
\]
which, by the argument in Example~3.4 of \cite{Fernandes}, is an integer multiple of $4\pi$. Since the coefficient space is
$\mathbb R/4\pi\mathbb Z$,
this difference vanishes. Hence
$ (\sigma\star\tau)\star\rho
=
\sigma\star(\tau\star\rho)$ ,
and therefore the multiplication on $Bis(G)$ is globally associative.

\end{example}

The converse of Theorem \ref{th3} holds for local Lie groupoids such that every arrow of $G$ is contained in a global bisection, i.e., for every $g \in G$, there exists a bisection $\sigma \in Bis(G)$  such that
$$
\sigma(\alpha(g)) = g.
$$
More precisely:

\begin{theorem}\label{th4}
Let $G$ be a locally convex, locally metrizable local Lie groupoid over a compact manifold $M$ admitting an adapted local addition such that every arrow of $G$ is contained in a global bisection. If $Bis(G)$  is globally associative, then $G$ is also globally associative.
\end{theorem}

\begin{proof}
Suppose that $Bis(G)$  is globally associative and let
$g_1,g_2,\dots,g_n \in G$
be elements whose product exists in $G$ . By assumption, for each $i=1,2,\dots,n$, there exists a bisection
$\sigma_i \in Bis(G)$ 
such that
$$
\sigma_i(\alpha(g_i)) = g_i.
$$

Let
$P(g_1,g_2,\dots,g_n)
\quad \text{and} \quad
Q(g_1,g_2,\dots,g_n)$
be any two defined parenthesized products in $G$. Since the multiplication on $Bis(G)$ is induced pointwise from that of $G$, evaluating the corresponding bisection products at the appropriate base point recovers the groupoid products. Hence,
\begin{align*}
P(g_1,g_2,\dots,g_n)
&=
P(\sigma_1,\sigma_2,\dots,\sigma_n)(\alpha(g_n)) \\
&=
Q(\sigma_1,\sigma_2,\dots,\sigma_n)(\alpha(g_n)) \\
&=
Q(g_1,g_2,\dots,g_n),
\end{align*}
where the second equality follows from the global associativity of $Bis(G)$. Therefore, $G$ is globally associative.
\end{proof}

\begin{Remark}
The following is an immediate consequence of Theorem \ref{Th2}, Theorem \ref{th4}, and Mal'cev's theorem for local Lie groups.

Let $G$ be a strongly connected local Lie groupoid satisfying the assumptions of Theorem \ref{th4} such that $ Bis(G)$ is connected. Then $G$ is globalizable if and only if $Bis(G)$ is globalizable.
\end{Remark}
The following theorem describes the globalization of $Bis(G)$  under additional assumptions on $G$ and its globalization.

\begin{theorem}
Let $G$ satisfy the assumptions of Theorem \ref{th3}. Assume further that $G$ is globalizable with globalization $\widetilde{G}$, where $\widetilde{G}$ is a locally convex, locally metrizable Lie groupoid over $M$ admitting an adapted local addition.
Then $Bis(G)$  embeds as an open local Lie subgroup of the Lie group $Bis(\widetilde{G})$. Diagrammatically, this  may be represented as
$$\begin{tikzcd}[column sep=large,row sep=large]
G \arrow[r, hook, "f"] \arrow[d, "Bis"']
& \widetilde{G} \arrow[d, "Bis"] \\
Bis(G) \arrow[r, hook, "Bis(f)"]
& Bis(\widetilde{G})
\end{tikzcd}$$ 

Consequently, $Bis(G)$  is globalizable.
\end{theorem}

\begin{proof}
    Since $G$ is globalizable, there exists a globalization morphism
$f:G\to \widetilde{G}$ ,
where $\widetilde{G}$  is a Lie groupoid and $f(G)$  is an open neighborhood of the unit section in $\widetilde{G}$. By Theorem \ref{th0} the map $Bis(f):Bis(G) \rightarrow Bis(\widetilde{G})$, $\sigma \mapsto  f \circ \sigma$ defines a local Lie group morphism. 
Since $f(G)$  is an open neighborhood of the unit section in $ \widetilde{G}$, we have
$$Bis(f)(Bis(G))=\{\tau\in Bis(\widetilde{G}) : \tau(M)\subseteq f(G)\}.$$
As $M$ is compact and $f(G)\subseteq \widetilde{G}$  is open, the set on the right-hand side is an open subset of $Bis(\widetilde{G})$  in the smooth compact-open topology. Therefore,
$Bis(G)$ 
embeds as an open local Lie subgroup of the Lie group
$Bis(\widetilde{G})$. Consequently, $Bis(G)$  is globalizable.
\end{proof}

\end{document}